\documentclass[11pt, a4paper]{amsart}

\usepackage{amsfonts,amsmath,amssymb, amscd, color} 
\usepackage{graphicx}
\usepackage[colorlinks,urlcolor=blue,linkcolor=blue,citecolor=blue]{hyperref}

\newtheorem{theorem}{Theorem}[section]
\newtheorem{lemma}[theorem]{Lemma}

\newtheorem{prop}[theorem]{Proposition}
\newtheorem{coro}[theorem]{Corollary}

\theoremstyle{definition}
\newtheorem{rem}[theorem]{Remark}

\newcommand\ZZ{\mathbb Z}
\newcommand\RR{\mathbb{R}}
\newcommand\Sp{\operatorname{Sp}}
\newcommand\St{\operatorname{St}}
\newcommand\LHS{\operatorname{LHS}}
\newcommand\RHS{\operatorname{RHS}}

\DeclareMathOperator{\id}{id}

\numberwithin{equation}{section}

\title[The Steinberg group as a quotient of a braid group]{A braid-like presentation of\\ 
the integral Steinberg group of type~$C_2$}

\author{Christian Kassel}
\address{Christian Kassel: 
Institut de Recherche Math\'e\-ma\-tique Avanc\'ee,
CNRS \& Universit\'e de Strasbourg,
7 rue Ren\'{e} Descartes, 67084 Strasbourg, France}
\email{kassel@math.unistra.fr}
\urladdr{www-irma.u-strasbg.fr/\raise-2pt\hbox{\~{}}kassel/}



\begin{document}

\begin{abstract}
We show that the Steinberg group~$\St(C_2,\ZZ)$ associated with the Lie type~$C_2$ and with integer coefficients
can be realized as a quotient of the braid group~$B_6$ by one relation. 
As an application we give a new braid-like presentation of the symplectic modular group~$\Sp_4(\ZZ)$.
\end{abstract}

\maketitle

\begin{flushright}
\emph{In memoriam amici\\ Patrick Dehornoy\\ (1952--2019)}
\end{flushright}

%%%
\section{Introduction}\label{sec-intro}

Let $B_6$ be the braid group on six strands. 
It has a standard presentation with five generators 
$\sigma_1$, $\sigma_2$, $\sigma_3$, $\sigma_4$, $\sigma_5$, 
and the following ten relations ($1 \leq i,j \leq 5$):
\begin{equation}\label{braid1}
\sigma_i \sigma_j  = \sigma_j  \sigma_i  \qquad\;\; \text{if} \; |i-j| > 1
\end{equation}
and 
\begin{equation}\label{braid2}
\sigma_i \sigma_j \sigma_i  = \sigma_j \sigma_i \sigma_j \quad \text{if} \; |i-j| = 1 .
\end{equation}

The purpose of this note is to show that if we add the single relation 
\begin{equation}\label{rel-b6st}
(\sigma_1 \sigma_2 \sigma_1)^2 (\sigma_1 \sigma_3^{-1} \sigma_5) (\sigma_1 \sigma_2 \sigma_1)^{-2} 
(\sigma_1 \sigma_3^{-1} \sigma_5) = 1 ,
\end{equation}
to Relations\,\eqref{braid1} and\,\eqref{braid2}, we obtain a presentation of the Steinberg group $\St(C_2,\ZZ)$ 
associated with the Lie type~$C_2$ over the ring of integers (see Theorem~\ref{thm-main}).
This group was defined for general commutative rings in~\cite{StM}.

If we further add the relation 
\begin{equation}\label{rel-b6sp}
(\sigma_1 \sigma_2 \sigma_1)^4 = 1,
\end{equation}
we obtain a presentation of the symplectic modular group~$\Sp_4(\ZZ)$
(see Corollary~\ref{coro-Sp}).
It is known that the Steinberg group~$\St(C_2,\ZZ)$ is a group extension of~$\Sp_4(\ZZ)$ with infinite cyclic kernel.

In order to prove these results we construct a surjective group homomorphism
$f: B_6 \to \St(C_2,\ZZ)$, show that $f$ vanishes on the normal subgroup~$N$ of~$B_6$ generated by
the element represented by the braid word in~\eqref{rel-b6st}, 
and construct a group homomorphism $\varphi: \St(C_2,\ZZ)\to B_6/N$ such that $\varphi\circ f = \id$.
The homomorphism $f$ is a lifting of a map $B_6 \rightarrow \Sp_4(\ZZ)$,
which is a special case of the homomorphism $\overline{f} : B_{2g+2} \rightarrow \Sp_{2g}(\ZZ)$
constructed in~\cite[Sect.\,4]{Ka} for any $g\geq 1$.

The note is organized as follows. 
In Section~\ref{sec-Steinberg} we give a presentation of the Steinberg group~$\St(C_2,\ZZ)$ 
and we prove a number of relations between special elements of~$\St(C_2,\ZZ)$. 
In Section~\ref{sec-B6Sp4} we construct the homomorphism~$f$ from the braid group to
the Steinberg group.
In Section~\ref{sec-Sp4B6} we state and prove our results.
In Appendix~\ref{app-St} we give a presentation of the Steinberg group~$\St(C_2,R)$ with coefficients in
an arbitrary commutative ring~$R$.

\subsection*{Notation} All the groups we consider are noted multiplicatively.
We denote their identity elements by~$1$ and we use brackets for the commutators:
\[
[x,y] = x y x^{-1} y^{-1} \,.
\]
Recall that $[y,x]  = [x,y]^{-1}$.

%%%
\section{The Steinberg group~$\St(C_2,\ZZ)$}\label{sec-Steinberg}

The positive roots of the root system~$C_2$ consist of four vectors $\alpha$, $\beta$,
$\alpha + \beta$ and $2\alpha + \beta$ of the Euclidean plane~$\RR^2$, 
where the roots $\alpha$ and $\alpha + \beta$ are of length~$\sqrt{2}$,
the roots $\beta$ and $2\alpha + \beta$ are of length~$2$, 
and $\alpha$ is orthogonal to~$\alpha + \beta$.
Together with $-\alpha$, $-\beta$, $-(\alpha + \beta)$ and $-(2\alpha + \beta)$, 
they form the root system~$\Phi$ of type~$C_2$ (see e.g. \cite{Bo, Ca}).

\subsection{Defining relations for the Steinberg group}\label{ssec-St}

We have the following presentation for the Steinberg group~$\St(C_2,\ZZ)$ of type~$C_2$
over the ring~$\ZZ$ of integers.

\begin{prop}\label{prop-St}
The Steinberg group~$\St(C_2,\ZZ)$ has a presentation with eight generators 
\begin{equation*}
x_{\alpha}, \; x_{\beta}, \; x_{\alpha + \beta}, \; x_{2\alpha + \beta} , \; 
x_{-\alpha}, \; x_{-\beta}, \; x_{-(\alpha + \beta)}, \; x_{-(2\alpha + \beta)}
\end{equation*}
and the following 24 relations:
\begin{equation}\label{rel-x1}
[x_{\alpha}, x_{2\alpha + \beta}] = [x_{\beta}, x_{\alpha + \beta}] = [x_{\beta}, x_{2\alpha + \beta}]
= [x_{\alpha + \beta}, x_{2\alpha + \beta}]  = 1, 
\end{equation}
\begin{equation}\label{rel-x2}
[x_{\alpha}, x_{- \beta}] = [x_{\beta}, x_{-\alpha}] = [x_{\beta}, x_{-(2\alpha + \beta)}]
= [x_{-\beta}, x_{2\alpha + \beta}] = 1 , 
\end{equation}
\begin{multline}\label{rel-x3}
[x_{-\alpha}, x_{-(2\alpha + \beta)}] = [x_{-\beta}, x_{-(\alpha + \beta)}] \\
= [x_{-\beta}, x_{-(2\alpha + \beta)}] 
=  [x_{-(\alpha + \beta)}, x_{-(2\alpha + \beta)}] = 1, 
\end{multline}
\begin{equation}\label{rel-x4}
[x_{\alpha}, x_{\beta}] = x_{\alpha + \beta} \,  x_{2\alpha + \beta}  = x_{2\alpha + \beta} \, x_{\alpha + \beta} \, , 
\end{equation}
\begin{equation}\label{rel-x5}
[x_{\alpha}, x_{\alpha + \beta}] = x_{2\alpha + \beta}^2 \, , 
\end{equation}
\begin{equation}\label{rel-x6a}
[x_{\alpha}, x_{-(\alpha + \beta)}] = x_{-\beta}^{-2} \, , 
\end{equation}
\begin{equation}\label{rel-x7}
[x_{\alpha}, x_{-(2\alpha + \beta)}] = x_{- \beta}\, x_{-(\alpha + \beta)}^{-1}  = x_{-(\alpha + \beta)}^{-1} \, x_{ -\beta} \, ,
\end{equation}
\begin{equation}\label{rel-x8}
[x_{\beta}, x_{-(\alpha + \beta)}] = x_{-\alpha} \, x_{-(2\alpha + \beta)}  =  x_{-(2\alpha + \beta)}\, x_{-\alpha} \, ,
\end{equation}
\begin{equation}\label{rel-x9}
[x_{\alpha + \beta}, x_{-\alpha}] = x_{\beta}^{-2}  , 
\end{equation}
\begin{equation}\label{rel-x10}
[x_{\alpha + \beta}, x_{-\beta}] = x_{\alpha} \, x_{2\alpha + \beta}^{-1}  =  x_{2\alpha + \beta}^{-1}\, x_{\alpha} \, ,
\end{equation}
\begin{equation}\label{rel-x11}
[x_{\alpha + \beta}, x_{-(2\alpha + \beta)}] = x_{-\alpha} x_{\beta}^{-1}  =   x_{\beta}^{-1}  x_{-\alpha} \, ,
\end{equation}
\begin{equation}\label{rel-x12}
[x_{2\alpha + \beta}, x_{-\alpha}] = x_{\alpha + \beta}^{-1} \, x_{ \beta}^{-1}  = x_{ \beta}^{-1}\, x_{\alpha + \beta}^{-1} \, ,
\end{equation}
\begin{equation}\label{rel-x13}
[x_{2\alpha + \beta}, x_{-(\alpha + \beta)}] = x_{\alpha} x_{-\beta} = x_{-\beta}x_{\alpha} \, ,
\end{equation}
\begin{equation}\label{rel-x14}
[x_{-\alpha}, x_{-\beta}] = x_{-(\alpha + \beta)}^{-1} \,  x_{-(2\alpha + \beta)} 
= x_{-(2\alpha + \beta)} \, x_{-(\alpha + \beta)}^{-1}  \, ,
\end{equation}
\begin{equation}\label{rel-x15}
[x_{-\alpha}, x_{-(\alpha + \beta)}] = x_{-(2\alpha + \beta)}^{-2} \, . 
\end{equation}
\end{prop}

Note that in view of Relations~\eqref{rel-x10}, \eqref{rel-x11} and~\eqref{rel-x14} 
the generators $x_{\alpha}$, $x_{-\alpha}$ and $x_{-(\alpha+ \beta)}$ 
can be expressed in terms of the five remaining generators
$x_{\beta}$, $x_{\alpha + \beta}$, $x_{2\alpha+ \beta}$, $x_{-\beta}$, $x_{-(2\alpha+ \beta)}$.

Proposition~\ref{prop-St} will be proved in Section~\ref{prop-St-pf} below.

\subsection{The surjection $\pi: \St(C_2,\ZZ) \to \Sp_4(\ZZ)$}\label{ssec-pi}

Recall that the symplectic modular group~$\Sp_4(\ZZ)$ is the group of automorphisms of
the free abelian rank~$4$ group~$\ZZ^4$ preserving the standard alternating form, 
namely the group of $4 \times 4$-matrices~$M$ with integer entries such that
\begin{equation*}
M^T 
\begin{pmatrix}
0 & I_2 \\
- I_2 & 0
\end{pmatrix}
M = 
\begin{pmatrix}
0 & I_2 \\
- I_2 & 0
\end{pmatrix} ,
\end{equation*}
where $M^T$ is the transpose of~$M$ and $I_2$ is the identity $2 \times 2$-matrix.

It is known that the group~$\Sp_4(\ZZ)$ is generated by the following eight matrices 
(see~\cite{Be1} or~\cite{HR}):
\begin{equation}\label{def-Xa}
X_{\alpha} =
\begin{pmatrix}
 1 & 1 & 0 & 0\\
 0 & 1 & 0 & 0\\
 0 & 0 & 1 & 0\\
 0 & 0 & -1 & 1
\end{pmatrix},
\quad
X_{-\alpha} =
\begin{pmatrix}
 1 & 0 & 0 & 0\\
 1 & 1 & 0 & 0\\
 0 & 0 & 1 & -1\\
 0 & 0 & 0 & 1
\end{pmatrix},
\end{equation}
\begin{equation}
X_{\beta} =
\begin{pmatrix}
 1 & 0 & 0 & 0\\
 0 & 1 & 0 & 1\\
 0 & 0 & 1 & 0\\
 0 & 0 & 0 & 1
\end{pmatrix},
\quad
X_{-\beta} =
\begin{pmatrix}
 1 & 0 & 0 & 0\\
 0 & 1 & 0 & 0\\
 0 & 0 & 1 & 0\\
 0 & 1 & 0 & 1
\end{pmatrix},
\end{equation}
\begin{equation}
X_{\alpha + \beta} =
\begin{pmatrix}
 1 & 0 & 0 & 1\\
 0 & 1 & 1 & 0\\
 0 & 0 & 1 & 0\\
 0 & 0 & 0 & 1
\end{pmatrix},
\quad
X_{-(\alpha + \beta)} =
\begin{pmatrix}
 1 & 0 & 0 & 0\\
 0 & 1 & 0 & 0\\
 0 & 1 & 1 & 0\\
 1 & 0 & 0 & 1
\end{pmatrix},
\end{equation}
\begin{equation}\label{def-X2ab}
X_{2\alpha + \beta} =
\begin{pmatrix}
 1 & 0 & 1 & 0\\
 0 & 1 & 0 & 0\\
 0 & 0 & 1 & 0\\
 0 & 0 & 0 & 1
\end{pmatrix},
\quad
X_{-(2\alpha + \beta)} =
\begin{pmatrix}
 1 & 0 & 0 & 0\\
 0 & 1 & 0 & 0\\
 1 & 0 & 1 & 0\\
 0 & 0 & 0 & 1
\end{pmatrix}.
\end{equation}
Observe that each matrix $X_{-\gamma}$ ($\gamma \in \Phi$) is the transpose of~$X_{\gamma}$.

It is easy to check that 
the matrices $X_{\gamma}$ ($\gamma \in \Phi$) satisfy the relations~\eqref{rel-x1}--\eqref{rel-x15} defining 
the Steinberg group~$\St(C_2,\ZZ)$.
Therefore there exists a unique homomorphism 
\begin{equation}\label{def-pi}
\pi: \St(C_2,\ZZ) \to \Sp_4(\ZZ)
\end{equation}
such that $\pi(x_{\gamma}) = X_{\gamma}$ for all $\gamma \in \Phi$.
Since the matrices $X_{\gamma}$ ($\gamma \in \Phi$) generate~$\Sp_4(\ZZ)$,
the homomorphism~$\pi$ is surjective.

\subsection{Proof of Proposition~\ref{prop-St}}\label{prop-St-pf}

By~\cite[Sect.\,3]{Be2} the Steinberg group $\St(C_2,\ZZ)$ has a presentation 
with the set of generators $\{ x_{\gamma} : \gamma \in \Phi\}$
subject to the following relations:
if $\gamma, \delta \in \Phi$ such that $\gamma + \delta \neq 0$, then
\begin{equation}\label{rel-x}
\left[ x_{\gamma}, x_{\delta} \right] = \prod\, x_{i\gamma+j\delta}^{c_{i,j}^{\gamma, \delta}}\, ,
\end{equation}
where $i$ and $j$ are positive integers such that $i\gamma+j\delta$ belongs to~$\Phi$ 
and the exponents $c_{i,j}^{\gamma, \delta}$ 
are integers depending only on the structure of~$\Sp_4(\ZZ)$.
In order to find the structure constants~$c_{i,j}^{\gamma, \delta}$ it is enough to 
apply the homomorphism~$\pi$ of~\eqref{def-pi} and to compute
the integers~$c_{i,j}^{\gamma, \delta}$ appearing in the relations
\begin{equation*}
\left[ X_{\gamma}, X_{\delta} \right] = \prod\, X_{i\gamma+j\delta}^{c_{i,j}^{\gamma, \delta}} \, ,
\end{equation*}
where $X_{\gamma} = \pi(x_{\gamma})$ are the elements of~$\Sp_4(\ZZ)$ defined by~\eqref{def-Xa}--\eqref{def-X2ab}.
 
In particular, when $X_{\gamma} X_{\delta} = X_{\delta} X_{\gamma}$ ($\gamma, \delta \in \Phi$), then 
$x_{\gamma} x_{\delta} = x_{\delta} x_{\gamma}$.
Relations~\eqref{rel-x1}--\eqref{rel-x3} follow from this remark.
We next have the following equalities in~$\Sp_4(\ZZ)$:
\begin{equation*}
[X_{\alpha}, X_{\beta}] = X_{\alpha + \beta} \,  X_{2\alpha + \beta} \,  , 
\quad
[X_{\alpha}, X_{\alpha + \beta}] = X_{2\alpha + \beta}^{2} \, , 
\end{equation*}
\begin{equation*}
[X_{\alpha}, X_{-(\alpha + \beta)}] = X_{-\beta}^{-2} \, , 
\quad
[X_{\alpha}, X_{-(2\alpha + \beta)}] = X_{- \beta}\, X_{-(\alpha + \beta)}^{-1} \, ,
\end{equation*}
\begin{equation*}
[X_{\beta}, X_{-(\alpha + \beta)}] = X_{-\alpha} \, X_{-(2\alpha + \beta)} \, ,
\quad
[X_{\alpha + \beta}, X_{-\alpha}] = X_{\beta}^{-2} , 
\end{equation*}
\begin{equation*}
[X_{\alpha + \beta}, X_{-\beta}] = X_{\alpha} \, X_{2\alpha + \beta}^{-1} \,  ,
\quad
[X_{\alpha + \beta}, X_{-(2\alpha + \beta)}] = X_{-\alpha} X_{\beta}^{-1} \, ,
\end{equation*}
\begin{equation*}
[X_{2\alpha + \beta}, X_{-\alpha}] = X_{\alpha + \beta}^{-1} \, X_{ \beta}^{-1} \,  ,
\quad[X_{2\alpha + \beta}, X_{-(\alpha + \beta)}] = X_{\alpha} X_{-\beta} \,  ,
\end{equation*}
\begin{equation*}
[X_{-\alpha}, X_{-\beta}] = X_{-(\alpha + \beta)}^{-1} \,  X_{-(2\alpha + \beta)} \,   ,
\quad
[X_{-\alpha}, X_{-(\alpha + \beta)}] = X_{-(2\alpha + \beta)}^{-2} \, . 
\end{equation*}
From these relations the remaining relations of Proposition~\ref{prop-St} follow.

\subsection{The elements $w_{\gamma}$}\label{ssec-Weyl}

For any root $\gamma \in \Phi$, set 
\begin{equation}\label{def-w}
w_{\gamma} = x_{\gamma} \, x_{-\gamma}^{-1} \, x_{\gamma} \, .
\end{equation}
In Steinberg's notation (see~\cite{St0, St}), 
we have $x_{\gamma} = x_{\gamma}(1)$ and $w_{\gamma} = w_{\gamma}(1)$.

In the sequel we need to know that the kernel of~$\pi: \St(C_2,\ZZ) \to \Sp_4(\ZZ)$Ê
is generated by~$w_{\gamma}^4$ for any long root~$\gamma$,
for instance by~$w_{\beta}^4$ or by~$w_{2\alpha + \beta}^4$;
this follows from~\cite[Kor.\,3.2]{Be2}.
The fact that such a generator of the kernel is of infinite order was established in~\cite[Th.\,6.3]{Ma}.

We now list a few properties of the elements~$w_{\gamma}$.

\begin{lemma}\label{lem-ww}
We have $w_{\gamma} = w_{-\gamma}^{-1}$ for all $\gamma \in \Phi$. 
\end{lemma}

\begin{proof}
By \cite{St0} (see also \cite[Lemme\,5.2\,(g)]{Ma}) we have
\begin{eqnarray*}
w_{\gamma} & = & w_{\gamma}(1) = w_{-\gamma}(-1) \\
& = & x_{-\gamma}(-1) x_{\gamma}(1) x_{-\gamma}(-1) 
= x_{-\gamma}^{-1} x_{\gamma} x_{-\gamma}^{-1} \\
&  = & \left( x_{-\gamma} x_{\gamma}^{-1} x_{-\gamma} \right)^{-1} 
= w_{-\gamma}^{-1} \, ,
\end{eqnarray*}
which was to be proved.
\end{proof}

\begin{lemma}\label{lem-wxw}
We have 
\begin{equation*}
w_{\beta} \, x_{\alpha} \, w_{\beta}^{-1} =  x_{\alpha + \beta}^{-1} \, ,\qquad
w_{\beta} \, x_{-\alpha} \, w_{\beta}^{-1} =  x_{-(\alpha + \beta)}^{-1} \, ,
\end{equation*}
\begin{equation*}
w_{\beta} \, x_{\alpha + \beta} \, w_{\beta}^{-1} = x_{\alpha} \, ,\qquad
w_{\beta} \, x_{-(\alpha + \beta)} \, w_{\beta}^{-1} = x_{-\alpha} \, ,
\end{equation*}
\begin{equation*}
w_{\beta} \, x_{2\alpha + \beta} \, w_{\beta}^{-1} = x_{2\alpha + \beta} \, ,\qquad
w_{\beta} \, x_{-(2\alpha + \beta)} \, w_{\beta}^{-1} = x_{-(2\alpha + \beta)} \, ,
\end{equation*}
\begin{equation*}
w _{2\alpha + \beta} \, x_{\beta} \, w _{2\alpha + \beta}^{-1} = x_{\beta} \, ,\qquad
w _{2\alpha + \beta} \, x_{-\beta} \, w _{2\alpha + \beta}^{-1} = x_{-\beta} \, , 
\end{equation*}
\begin{equation*}
w _{2\alpha + \beta} \, x_{\alpha} \, w _{2\alpha + \beta}^{-1} = x_{-(\alpha + \beta)}^{-1} \, ,\qquad
w _{2\alpha + \beta} \, x_{-(\alpha + \beta)} \, w _{2\alpha + \beta}^{-1} = x_{\alpha} \, , 
\end{equation*}
\begin{equation*}
w _{2\alpha + \beta} \, x_{\alpha + \beta} \, w _{2\alpha + \beta}^{-1} = x_{-\alpha} \, , \qquad
w _{2\alpha + \beta} \, x_{-\alpha} \, w _{2\alpha + \beta}^{-1} = x_{\alpha + \beta}^{-1} \, ,
\end{equation*}
\begin{equation*}
w _{\alpha} \, x_{\beta} \, w _{\alpha}^{-1} = x_{2\alpha + \beta} \, , \qquad
w _{\alpha} \, x_{-\beta} \, w _{\alpha}^{-1} = x_{-(2\alpha + \beta)} \, ,
\end{equation*}
\begin{equation*}
w _{\alpha} \, x_{\alpha + \beta} \, w _{\alpha}^{-1} = x_{\alpha + \beta}^{-1} \, , \qquad
w _{\alpha} \, x_{-(\alpha + \beta)} \, w _{\alpha}^{-1} = x_{-(\alpha + \beta)}^{-1} \, ,
\end{equation*}
\begin{equation*}
w _{\alpha} \, x_{2\alpha + \beta} \, w _{\alpha}^{-1} = x_{\beta} \, , \qquad
w _{\alpha} \, x_{-(2\alpha + \beta)} \, w _{\alpha}^{-1} = x_{-\beta} \, ,
\end{equation*}
\begin{equation*}
w _{\alpha + \beta} \, x_{\alpha} \, w _{\alpha + \beta}^{-1} = x_{\alpha}^{-1} \, ,\qquad
w _{\alpha + \beta} \, x_{-\alpha} \, w _{\alpha + \beta}^{-1} = x_{-\alpha}^{-1} \, , 
\end{equation*}
\begin{equation*}
w _{\alpha + \beta} \, x_{\beta} \, w _{\alpha + \beta}^{-1} = x_{-(2\alpha+ \beta)}^{-1} \, ,\qquad
w _{\alpha + \beta} \, x_{-(2\alpha+ \beta)} \, w _{\alpha + \beta}^{-1} = x_{\beta}^{-1} \, , 
\end{equation*}
\begin{equation*}
w _{\alpha + \beta} \, x_{2\alpha+ \beta} \, w _{\alpha + \beta}^{-1} = x_{-\beta}^{-1} \, ,\qquad
w _{\alpha + \beta} \, x_{-\beta} \, w _{\alpha + \beta}^{-1} =x_{2\alpha+ \beta}^{-1} \, , 
\end{equation*}
\end{lemma}

\begin{proof}
By Relation\,(R7) in~\cite[Chap.\,3, p.\,23]{St}, for any two roots $\gamma$ and $\delta$ such that
$\gamma + \delta \neq 0$ we have
$w_{\gamma} \, x_{\delta} \, w_{\gamma}^{-1} = x_{\delta'}^{\varepsilon}$,
where $\delta'$ is the image of~$\delta$ under the reflection in the line orthogonal to~$\gamma$ and
$\varepsilon = \pm 1$. To determine the root~$\delta'$ and the sign~$\varepsilon$ 
it is enough to  compute the image 
$\pi(w_{\gamma} \, x_{\delta} \, w_{\gamma}^{-1})$ in~$\Sp_4(\ZZ)$.
\end{proof}

%%%
\section{From the braid group to the Steinberg group}\label{sec-B6Sp4}

We now construct a homomorphism from the braid group~$B_6$ to the Steinberg group~$\St(C_2,\ZZ)$.

\begin{prop}\label{prop-BtoSt}
There exists a unique homomorphism 
$f: B_6 \to \St(C_2,\ZZ)$ such that 
\begin{equation*}\label{defmapf12}
f(\sigma_1) = x_{2\alpha + \beta} \, , \qquad 
f(\sigma_2) = x_{-(2\alpha + \beta)}^{-1} \, ,
\end{equation*}
\begin{equation*}\label{defmapf3}
f(\sigma_3) = x_{\beta} \, x_{\alpha + \beta}^{-1} \, x_{2\alpha + \beta} \, ,
\end{equation*}
\begin{equation*}\label{defmapf45}
f(\sigma_4) = x_{-\beta}^{-1} \, , \qquad 
f(\sigma_5) = x_{\beta} \, .
\end{equation*}
The homomorphism $f$ is surjective.
\end{prop}

The homomorphism~$f$ lifts the homomorphism $\overline{f} : B_6 \to \Sp_4(\ZZ)$ 
constructed in~\cite[Sect.\,4.2]{Ka} in the sense that $\overline{f} = \pi \circ f$, where
$\pi$ is the natural surjection $\St(C_2,\ZZ) \to \Sp_4(\ZZ)$ defined in\,\eqref{def-pi}.

\begin{proof}
For the existence and the uniqueness of~$f$
it suffices to check that the five elements $f(\sigma_i)$ ($1 \leq i \leq 5$) 
of the Steinberg group~$\St(C_2,\ZZ)$ satisfy the braid relations~\eqref{braid1} and~\eqref{braid2}.

(i) Let us first check the commutation relations~\eqref{braid1}.

\begin{itemize}
\item
\emph{Commutation of~$f(\sigma_1)$ with $f(\sigma_3)$, $f(\sigma_4)$ and $f(\sigma_5)$.}
This follows from the fact that 
$x_{2\alpha + \beta}$ commutes with $x_{\beta}$ and $x_{\alpha + \beta}$ by\,\eqref{rel-x1}, and
with $x_{-\beta}$ by\,\eqref{rel-x2}.

\item 
\emph{Commutation of~$f(\sigma_2)$ with $f(\sigma_4)$ and $f(\sigma_5)$.}
It follows from\,\eqref{rel-x2} and\,\eqref{rel-x3}.

\item
\emph{Commutation of~$f(\sigma_3)$ with $f(\sigma_5)$.}
It follows from\,\eqref{rel-x1}.
\end{itemize}

(ii) The relation $f(\sigma_1) f(\sigma_2) f(\sigma_1) = f(\sigma_2) f(\sigma_1) f(\sigma_2)$ reads as
\begin{equation*}
x_{2\alpha + \beta} \, x_{-(2\alpha + \beta)}^{-1} \, x_{2\alpha + \beta}  
= x_{-(2\alpha + \beta)}^{-1} \, x_{2\alpha + \beta} \, x_{-(2\alpha + \beta)}^{-1}\, ,
\end{equation*}
which is equivalent to
$w_{2\alpha + \beta}  = w_{-(2\alpha + \beta)}^{-1}$. The latter holds by Lemma~\ref{lem-ww}.

(iii) The relation $f(\sigma_2) f(\sigma_3) f(\sigma_2) = f(\sigma_3) f(\sigma_2) f(\sigma_3)$ reads as
\begin{equation*}\label{rel-232}
x_{-(2\alpha + \beta)}^{-1} \, x_{\beta} \, x_{\alpha + \beta}^{-1} \, x_{2\alpha + \beta} \, x_{-(2\alpha + \beta)}^{-1} 
= x_{\beta} \, x_{\alpha + \beta}^{-1} \, x_{2\alpha + \beta} \, x_{-(2\alpha + \beta)}^{-1} \, 
x_{\beta} \, x_{\alpha + \beta}^{-1} \, x_{2\alpha + \beta} \, .
\end{equation*}
Let $\LHS$ (resp.\ $\RHS$) be the element of~$\St(C_2,\ZZ)$ represented by the left-hand (resp.\ right-hand) side of the 
previous equation.
By \eqref{rel-x1}, \eqref{rel-x2}, \eqref{rel-x3}, \eqref{rel-x11}, \eqref{def-w} we obtain
\begin{eqnarray*}
\LHS 
& = & x_{\beta} \, \underbrace{x_{-(2\alpha + \beta)}^{-1} \,  x_{\alpha + \beta}^{-1}} \, 
x_{-(2\alpha + \beta)} \, w_{-(2\alpha + \beta)}^{-1} \\
& = & x_{\beta} \,   x_{\alpha + \beta}^{-1} \, \underbrace{x_{-(2\alpha + \beta)}^{-1} \, x_{-\alpha}^{-1} \, x_{\beta} \, 
x_{-(2\alpha + \beta)}} \, w_{-(2\alpha + \beta)}^{-1} \\
& = & \underbrace{x_{\beta} \,   x_{\alpha + \beta}^{-1} \,  x_{-\alpha}^{-1} \, x_{\beta}} \, w_{-(2\alpha + \beta)}^{-1} \\
& = & x_{\beta}^2 \,   x_{\alpha + \beta}^{-1} \,  x_{-\alpha}^{-1} \, w_{-(2\alpha + \beta)}^{-1} \, .
\end{eqnarray*}
(The above underbraces ${\underbrace{}}$ mark the places to which we apply the relations we refer to.)

Let us now deal with~$\RHS$.
Since by\,\eqref{rel-x1} and\,\eqref{rel-x2} $x_{\beta}$ commutes with $x_{\alpha + \beta}$ and with $x_{\pm(2\alpha + \beta)}$,
and $x_{\alpha + \beta}$ commutes with $x_{2\alpha + \beta}$, we obtain
\begin{equation*}
\RHS = x_{\beta}^2 \, x_{\alpha + \beta}^{-1} \, w_{2\alpha + \beta}  \, x_{\alpha + \beta}^{-1} \, .
\end{equation*}
Now by Lemma~\ref{lem-wxw}, we have $w_{2\alpha + \beta}  \, x_{\alpha + \beta}^{-1} = x_{-\alpha}^{-1} \, w_{2\alpha + \beta}$.
Therefore,
\begin{equation*}
\RHS = x_{\beta}^2 \, x_{\alpha + \beta}^{-1} \, x_{-\alpha}^{-1} \, w_{2\alpha + \beta} \, .
\end{equation*}
It follows that $\LHS = \RHS$ is equivalent to $w_{-(2\alpha + \beta)}^{-1} = w_{2\alpha + \beta}$, 
which again holds by Lemma~\ref{lem-ww}.

(iv) The relation $f(\sigma_3) f(\sigma_4) f(\sigma_3) = f(\sigma_4) f(\sigma_3) f(\sigma_4)$ reads as
\begin{equation*}
x_{\beta} \, x_{\alpha + \beta}^{-1} \, x_{2\alpha + \beta} \, x_{-\beta}^{-1} \, x_{\beta} \, x_{\alpha + \beta}^{-1} \, x_{2\alpha + \beta}
=  x_{-\beta}^{-1} \, x_{\beta} \, x_{\alpha + \beta}^{-1} \, x_{2\alpha + \beta} \, x_{-\beta}^{-1} \, .
\end{equation*}
Let $\LHS'$ (resp.\ $\RHS'$) be the element of~$\St(C_2,\ZZ)$ represented by the left-hand (resp.\ right-hand) side of the 
previous equation.
Since by \eqref{rel-x1} and \eqref{rel-x2}
$x_{2\alpha + \beta}$ commutes with $x_{\alpha}$, with $x_{\alpha + \beta}$ and with~$x_{\pm\beta}$, hence with~$w_{\beta}$, 
and $x_{\beta}$ commutes with $x_{\alpha + \beta}$, we obtain
\begin{equation*}\label{eq-LHS}
\LHS' 
=   x_{\alpha + \beta}^{-1} \, x_{2\alpha + \beta}^2\, w_{\beta} \, x_{\alpha + \beta}^{-1}
= x_{\alpha + \beta}^{-1} \, x_{2\alpha + \beta}^2\, x_{\alpha}^{-1} \,  w_{\beta}
= x_{\alpha + \beta}^{-1} \, x_{\alpha}^{-1} \,  x_{2\alpha + \beta}^2  \,  w_{\beta} \, ,
\end{equation*}
the second equality holding by Lemma~\ref{lem-wxw}.
For $\RHS'$, by Lemma~\ref{lem-ww} we have
\begin{equation*}
\RHS'
= w_{\beta} \, x_{-\beta}\, x_{\alpha + \beta}^{-1} \, x_{-\beta}^{-1} \, x_{2\alpha + \beta} \, .
\end{equation*}
Now by\,\eqref{rel-x10}, we have
$x_{-\beta}\, x_{\alpha + \beta}^{-1} \, x_{-\beta}^{-1} = x_{\alpha + \beta}^{-1} \, x_{\alpha} \, x_{2\alpha + \beta}^{-1}$.
Therefore,
\begin{equation*}
\RHS'
= w_{\beta} \, x_{\alpha + \beta}^{-1} \, x_{\alpha} \,  .
\end{equation*}
It follows from Lemma~\ref{lem-wxw} that $w_{\beta} \, x_{\alpha + \beta}^{-1} =  x_{\alpha}^{-1} \, w_{\beta}$
and $w_{\beta} \, x_{\alpha} =  x_{\alpha + \beta}^{-1} \, w_{\beta}$.
Moreover by\,\eqref{rel-x5} we have 
$x_{\alpha}^{-1} \, x_{\alpha + \beta}^{-1} = x_{\alpha + \beta}^{-1} \, x_{\alpha}^{-1}  \, x_{2\alpha + \beta}^2$.
Therefore,
\begin{equation*}\label{eq-RHS}
\RHS'
= x_{\alpha}^{-1} \, w_{\beta} \, x_{\alpha} 
= x_{\alpha}^{-1} \, x_{\alpha + \beta}^{-1} \, w_{\beta}
= x_{\alpha + \beta}^{-1} \, x_{\alpha}^{-1}  \, x_{2\alpha + \beta}^2 \, w_{\beta} = {\LHS}'.
\end{equation*}

(v) The relation $f(\sigma_4) f(\sigma_5) f(\sigma_4) = f(\sigma_4) f(\sigma_5) f(\sigma_4)$ reads as
\begin{equation*}
x_{-\beta}^{-1} \,  x_{\beta} \, x_{-\beta}^{-1} = x_{\beta} \, x_{-\beta}^{-1} \,  x_{\beta} \, ,
\end{equation*}
which is equivalent to $w_{-\beta}^{-1} = w_{\beta}$. The last equality holds by Lemma~\ref{lem-ww}.

As we noted after stating Proposition~\ref{prop-St}, the five elements
$x_{\beta}$, $x_{\alpha + \beta}$, $x_{2\alpha+ \beta}$, $x_{-\beta}$, $x_{-(2\alpha+ \beta)}$
generate~$\St(C_2,\ZZ)$. 
They clearly are in the image of~$f$, which implies the surjectivity of the latter.
\end{proof}

%%%
\section{Results}\label{sec-Sp4B6}

Let $f: B_6 \to \St(C_2,\ZZ)$ be the surjective homomorphism defined in Proposition~\ref{prop-BtoSt}.
We now state our main result. 

\begin{theorem}\label{thm-main}
The kernel of the homomorphism $f: B_6 \to \St(C_2,\ZZ)$ is the normal subgroup of~$B_6$ generated by 
\begin{equation*}
(\sigma_1 \sigma_2 \sigma_1)^2 (\sigma_1 \sigma_3^{-1} \sigma_5) (\sigma_1 \sigma_2 \sigma_1)^{-2} 
(\sigma_1 \sigma_3^{-1} \sigma_5) .
\end{equation*}
\end{theorem}

This means that the Steinberg group~$\St(C_2,\ZZ)$ has a presentation with five generators 
$\sigma_1$, $\sigma_2$, $\sigma_3$, $\sigma_4$, $\sigma_5$, and eleven relations 
consisting of the ten braid relations~\eqref{braid1}--\eqref{braid2} and the additional relation
\begin{equation}\label{eq-relSt}
(\sigma_1 \sigma_2 \sigma_1)^2 (\sigma_1 \sigma_3^{-1} \sigma_5) (\sigma_1 \sigma_2 \sigma_1)^{-2} 
= (\sigma_1 \sigma_3^{-1} \sigma_5)^{-1} .
\end{equation}
This relation is clearly equivalent to~\eqref{rel-b6st}.

As mentioned at the beginning of Section~\ref{ssec-Weyl},
the kernel of the projection $\pi: \St(C_2,\ZZ) \to \Sp_4(\ZZ)$ is generated by~$w_{2\alpha+\beta}^4$.
Since 
\begin{equation}\label{eq-w121}
w_{2\alpha+\beta} = x_{2\alpha+\beta} x_{-(2\alpha+\beta)}^{-1} x_{2\alpha+\beta}
= f(\sigma_1 \sigma_2 \sigma_1) ,
\end{equation}
we deduce the following presentation\footnote{Group presentations of~$\Sp_4(\ZZ)$ of a different kind
have been given in~\cite{Be1, Bn}.} of~$\Sp_4(\ZZ)$.

\begin{coro}\label{coro-Sp}
The symplectic modular group~$\Sp_4(\ZZ)$ has a presentation with five generators 
$\sigma_1$, $\sigma_2$, $\sigma_3$, $\sigma_4$, $\sigma_5$, and twelve relations 
consisting of Relations~\eqref{braid1}--\eqref{braid2} and the two relations
\begin{equation*}
(\sigma_1 \sigma_2 \sigma_1)^2 (\sigma_1 \sigma_3^{-1} \sigma_5) (\sigma_1 \sigma_2 \sigma_1)^{-2} 
= (\sigma_1 \sigma_3^{-1} \sigma_5)^{-1}  
\end{equation*}
and
\begin{equation*}
(\sigma_1 \sigma_2 \sigma_1)^4 =  1 .
\end{equation*}
\end{coro}

Let $N$ be the normal subgroup of~$B_6$ generated by 
\begin{equation*}
\beta = (\sigma_1 \sigma_2 \sigma_1)^2 (\sigma_1 \sigma_3^{-1} \sigma_5) (\sigma_1 \sigma_2 \sigma_1)^{-2} 
(\sigma_1 \sigma_3^{-1} \sigma_5) .
\end{equation*}
Theorem~\ref{thm-main} is a consequence of the following two propositions.

\begin{prop}\label{prop-vanishing}
We have $f(N) = 1$.
\end{prop}

\begin{proof}
It suffices to check that $f(\beta) = 1$. By~\eqref{eq-w121} we have
$f(\sigma_1 \sigma_2 \sigma_1) = w_{2\alpha+\beta}$.
We also have
\begin{equation*}
f(\sigma_1 \sigma_3^{-1} \sigma_5) 
= x_{2\alpha+\beta} (x_{2\alpha+\beta}^{-1} x_{\alpha+\beta} x_{\beta}^{-1}) x_{\beta} = x_{\alpha+\beta}\, .
\end{equation*}
Therefore, using Lemma~\ref{lem-wxw}, we obtain
\begin{equation*}
f(\beta) 
= w_{2\alpha+\beta}^2 x_{\alpha+\beta} w_{2\alpha+\beta}^{-2} x_{\alpha+\beta}
= w_{2\alpha+\beta} x_{-\alpha} w_{2\alpha+\beta}^{-1} x_{\alpha+\beta}
= x_{\alpha+\beta}^{-1} x_{\alpha+\beta} = 1 .
\end{equation*}
\end{proof}

It follows from the previous proposition
that $f: B_6 \to \St(C_2,\ZZ)$ factors through a homomorphism $B_6/N \to \St(C_2,\ZZ)$,
which we still denote by~$f$.

\begin{prop}\label{prop-phi}
There exists a homomorphism $\varphi: \St(C_2,\ZZ) \to B_6/N$ such that 
\begin{eqnarray*}
\varphi(x_{\alpha}) & = & (\sigma_5 \sigma_4) (\sigma_1 \sigma^{-1}_3 \sigma_5) (\sigma_5 \sigma_4)^{-1}  ,  \\
\varphi(x_{\beta}) & = & \sigma_5 \, ,  \\
\varphi(x_{\alpha + \beta}) & = & \sigma_1 \sigma_3^{-1} \sigma_5  \, ,  \\
\varphi(x_{2\alpha + \beta}) & = & \sigma_1  \, ,  \\
\varphi(x_{-\alpha}) & = &  (\sigma_1 \sigma_2) (\sigma_1 \sigma^{-1}_3 \sigma_5) (\sigma_1 \sigma_2)^{-1}   , \\
\varphi(x_{-\beta}) & = & \sigma^{-1}_4  ,   \\
\varphi(x_{-(\alpha + \beta)}) 
& = & (\sigma_1 \sigma_2 \sigma_ 5 \sigma_4) (\sigma^{-1}_1 \sigma_3 \sigma_5^{-1})
(\sigma_1 \sigma_2 \sigma_ 5 \sigma_4)^{-1}  ,  \\
\varphi(x_{-(2\alpha + \beta)}) & = & \sigma^{-1}_2  
\end{eqnarray*}
modulo~$N$.
Moreover, $\varphi \circ f = \id$.
\end{prop}

Before we prove Proposition~\ref{prop-phi}, let us record the following four equalities in the braid group~$B_6$.
One may check them using one's favorite algorithm for solving the word problem in
braid groups (see for instance the monographs \cite{Bi, De, DDRW, E++, KT}).

\begin{lemma}\label{lem-B6}
The following equalities hold in~$B_6$:
\begin{multline*}
[\varphi(x_{\alpha}), \varphi(x_{\beta})]^{-1} \varphi(x_{\alpha + \beta}) \, \varphi(x_{2\alpha + \beta}) \\
= (\sigma_4 \sigma_5 \sigma_4)^2 (\sigma_1\sigma_3^{-1} \sigma_5) 
(\sigma_4 \sigma_5 \sigma_4)^{-2} (\sigma_1 \sigma_3^{-1} \sigma_5) ,
\end{multline*}
\begin{equation*}
[\varphi(x_{\alpha}), \varphi(x_{\alpha + \beta})]^{-1} \varphi(x_{2\alpha + \beta}^2)
= (\sigma_1 \sigma_3^{-1} \sigma_5) (\sigma_4 \sigma_5 \sigma_4)^2 (\sigma_1 \sigma_3^{-1} \sigma_5)
(\sigma_4  \sigma_5  \sigma_4)^{-2}  ,
\end{equation*}
\begin{multline*}
[\varphi(x_{\alpha}), \varphi(x_{-(\alpha + \beta)})] \, \varphi(x_{-\beta}^2)\\
= (\sigma_5 \sigma_4 \sigma_1 \sigma_2) ( \sigma_1\sigma_2 \sigma_3)\, 
\gamma^{-1} \, ( \sigma_1\sigma_2 \sigma_3)^{-1} (\sigma_5 \sigma_4 \sigma_1\sigma_2)^{-1} ,
\end{multline*}
where $\gamma = (\sigma_1 \sigma_3^{-1} \sigma_5)
(\sigma_1 \sigma_2 \sigma_1)^{-2}  ( \sigma_1 \sigma_3^{-1} \sigma_5 )  
(\sigma_1 \sigma_2 \sigma_1)^2  $, 
and
\begin{equation*}
[\varphi(x_{\alpha + \beta}), \varphi(x_{-\alpha})] \, \varphi(x_{\beta}^2) 
= (\sigma_1 \sigma_3^{-1} \sigma_5) (\sigma_1 \sigma_2 \sigma_1)^2 (\sigma_1 \sigma_3^{-1} \sigma_5) 
(\sigma_1 \sigma_2 \sigma_1)^{-2} .
\end{equation*}
\end{lemma}

We make also the following observation. Let 
\begin{equation*}
\Delta = (\sigma_ 1\sigma_2 \sigma_3 \sigma_4 \sigma_5) (\sigma_1 \sigma_2 \sigma_3 \sigma_4) 
(\sigma_1 \sigma_2 \sigma_3) (\sigma_1 \sigma_2) \sigma_1
\end{equation*}
be the `half-twist' in the braid group~$B_6$. It has the following important property
(see for instance~\cite{Bi} or~\cite[Sect.\,1.3.3]{KT}):
\begin{equation}\label{rel-Delta1}
\Delta \sigma_i  \Delta^{-1} = \sigma_{6-i}  \, , \qquad (i = 1, \ldots, 5)
\end{equation}
which implies that its square $\Delta^2$ is central (actually $\Delta^2$ generates the center of
the braid group).
Note that the braids $\sigma_1 \sigma_5$, $\sigma_2 \sigma_4$ and $\sigma_3$
are invariant under conjugation by~$\Delta$ so that
\begin{equation*}\label{rel-Delta2}
\Delta (\sigma_1 \sigma_3^{-1} \sigma_5)  \Delta^{-1} = \sigma_5 \sigma_3^{-1} \sigma_1
= \sigma_1 \sigma_3^{-1} \sigma_5 \, ,
\end{equation*}
whereas $\sigma_1 \sigma_2 \sigma_1$ and $\sigma_4 \sigma_5 \sigma_4$ are interchanged:
\begin{equation}\label{rel-Delta3}
\Delta (\sigma_1 \sigma_2 \sigma_1)  \Delta^{-1} = \sigma_5 \sigma_4 \sigma_5
= \sigma_4 \sigma_5 \sigma_4 \, . 
\end{equation}

In view of \eqref{rel-Delta1} and of the definition of~$\varphi$ we have 
the following ``symmetries'':
\begin{equation}\label{sym1}
\Delta \, \varphi(x_{\beta}) \, \Delta^{-1}  = \varphi(x_{2\alpha + \beta})\, ,
\end{equation}
\begin{equation}\label{sym2}
\Delta \, \varphi(x_{-\beta}) \, \Delta^{-1} = \varphi(x_{-(2\alpha + \beta)}) \, ,
\end{equation}
\begin{equation}\label{sym3}
\Delta \, \varphi(x_{\alpha}) \, \Delta^{-1} = \varphi(x_{-\alpha})  \, ,
\end{equation}
\begin{equation}\label{sym4}
\Delta \, \varphi(x_{\pm(\alpha + \beta)}) \, \Delta^{-1} = \varphi(x_{\pm(\alpha + \beta)})  \, .
\end{equation}
Geometrically speaking, conjugating the elements $\varphi(x_{\gamma})$ ($\gamma\in \Phi$) 
by~$\Delta$ corresponds in the root system~$C_2$ to the reflection in the one-dimensional vector space 
spanned by~$\alpha + \beta$.

\begin{rem}\label{rem-Delta}
It follows from~\cite[Th.\,0.2]{BDM} that the centralizer of~$\Delta$ is an Artin group of type~$C_3$.
Actually it is generated by the above-mentioned elements $\beta_1 = \sigma_1 \sigma_5$, 
$\beta_2 = \sigma_2 \sigma_4$ and $\beta_3 = \sigma_3$, which satisfy the relations 
\begin{equation*}
\beta_1 \beta_2 \beta_1 = \beta_2 \beta_1 \beta_2 \, , \quad
\beta_1 \beta_3 = \beta_3 \beta_1 \, , \quad
\beta_2 \beta_3 \beta_2 \beta_3 = \beta_3 \beta_2 \beta_3 \beta_2 \, .
\end{equation*}
\end{rem}

\begin{proof}[Proof of Proposition~\ref{prop-phi}]
To prove the existence of $\varphi: \St(C_2,\ZZ) \to B_6/N$
it suffices to check that the elements $\varphi(x_{\gamma})$ ($\gamma \in \Phi$) satisfy the
24~relations~\eqref{rel-x1}--\eqref{rel-x15} in the quotient group~$B_6/N$.

(a) Relations~\eqref{rel-x1}: in view of the braid relations\,\eqref{braid1},
we have
\begin{equation*}
[\varphi(x_{\alpha}), \varphi(x_{2\alpha + \beta})] 
= [(\sigma_5 \sigma_4) (\sigma_1 \sigma^{-1}_3 \sigma_5) (\sigma_5 \sigma_4)^{-1}, \sigma_1] = 1\, ,
\end{equation*}
\begin{equation*}
[\varphi(x_{\beta}), \varphi(x_{\alpha + \beta})] = 
[\sigma_5, \sigma_1 \sigma_3^{-1} \sigma_5] = 1 \, ,
\end{equation*}
\begin{equation*}
[\varphi(x_{\beta}), \varphi(x_{2\alpha + \beta}) ] =  [\sigma_5, \sigma_1] = 1 \, ,
\end{equation*}
\begin{equation*}
[\varphi(x_{\alpha + \beta}), \varphi(x_{2\alpha + \beta}) ] =  [\sigma_1 \sigma_3^{-1} \sigma_5, \sigma_1] = 1 \, .
\end{equation*}

(b) Relations~\eqref{rel-x2}:
again in view of the commutation relations\,\eqref{braid1}, but also of the braid relations\,\eqref{braid2}, we have
\begin{eqnarray*}
[\varphi(x_{\alpha}), \varphi(x_{-\beta}) ] 
& = & 
[(\sigma_5 \sigma_4) (\sigma_1 \sigma^{-1}_3 \sigma_5) (\sigma_5 \sigma_4)^{-1}, \sigma^{-1}_4] \\
& =& \sigma_5 \sigma_4 (\sigma_1 \sigma^{-1}_3 \sigma_5) \underbrace{\sigma_4^{-1} \sigma_5^{-1} \sigma^{-1}_4}
 \sigma_5 \sigma_4 (\sigma_1 \sigma^{-1}_3 \sigma_5)^{-1}\underbrace{ \sigma_4^{-1} \sigma_5^{-1} \sigma_4} \\
 & =& \sigma_5 \sigma_4 (\sigma_1 \sigma^{-1}_3 \sigma_5) \underbrace{\sigma_5^{-1} 
 \underbrace{\sigma_4^{-1} \sigma^{-1}_5 \sigma_5 \sigma_4} 
 (\sigma_1 \sigma^{-1}_3 \sigma_5)^{-1} \sigma_5} \sigma_4^{-1} \sigma_5^{-1} \\
& = & \sigma_5 \sigma_4 (\sigma_1 \sigma^{-1}_3 \sigma_5)(\sigma_1 \sigma^{-1}_3 \sigma_5)^{-1}
\sigma_4^{-1} \sigma_5^{-1} = 1.
\end{eqnarray*}
(Here and below the underbraces indicate the places where we apply the braid relations
and the trivial relations $\sigma_i \sigma_i^{-1} = \sigma_i^{-1} \sigma_i = 1$.)

Similarly, by~\eqref{braid1} we have
\begin{equation*}
[\varphi(x_{\beta}), \varphi(x_{-\alpha})] = 
[\sigma_5,  (\sigma_1 \sigma_2) (\sigma_1 \sigma^{-1}_3 \sigma_5) (\sigma_1 \sigma_2)^{-1}] = 1\, ,
\end{equation*}
\begin{equation*}
[\varphi(x_{\beta}), \varphi(x_{-(2\alpha + \beta)}) ] = [\sigma_5, \sigma^{-1}_2] = 1 \, ,
\end{equation*}
\begin{equation*}
[\varphi(x_{-\beta}), \varphi(x_{2\alpha + \beta}) ] = [\sigma^{-1}_4, \sigma_1] = 1 \, .
\end{equation*}
Note that only $[\varphi(x_{\alpha}), \varphi(x_{-\beta}) ] = 1$ requires one of the braid relations~\eqref{braid2}.

(c) Relations~\eqref{rel-x3}: 
for the first one, using \eqref{sym2} and \eqref{sym3}, we have
\begin{eqnarray*}
[\varphi(x_{-\alpha}), \varphi(x_{-(2\alpha + \beta)})] 
& = & [\Delta \, \varphi(x_{\alpha}) \, \Delta^{-1}, \Delta \, \varphi(x_{-\beta}) \, \Delta^{-1}]  \\
& = & \Delta \, [\varphi(x_{\alpha}) , \varphi(x_{-\beta})] \, \Delta^{-1}  = 1 
\end{eqnarray*}
since $[\varphi(x_{\alpha}), \varphi(x_{-\beta}) ] = 1$, as we have just proved in Item\,(b).

For the second one we use the equality 
$\varphi(x_{-(\alpha + \beta)}) = (\sigma_1 \sigma_2) \, \varphi(x_{\alpha})^{-1} (\sigma_1 \sigma_2)^{-1}$,
which we derive from the definition of~$\varphi$.
Then 
\begin{eqnarray*}
[\varphi(x_{-\beta}), \varphi(x_{-(\alpha + \beta)}]
& = & [\varphi(x_{-\beta}), (\sigma_1 \sigma_2) \, \varphi(x_{\alpha})^{-1} (\sigma_1 \sigma_2)^{-1}] \\
& = & [\sigma_4^{-1}, (\sigma_1 \sigma_2) \, \varphi(x_{\alpha})^{-1} (\sigma_1 \sigma_2)^{-1}] \\
& = & [(\sigma_1 \sigma_2) \sigma_4^{-1} (\sigma_1 \sigma_2)^{-1}, 
(\sigma_1 \sigma_2) \, \varphi(x_{\alpha})^{-1} (\sigma_1 \sigma_2)^{-1}] \\
& = & (\sigma_1 \sigma_2)\,  [\varphi(x_{-\beta}), \varphi(x_{\alpha})^{-1}] \, (\sigma_1 \sigma_2)^{-1}.
\end{eqnarray*}
Since $\varphi(x_{-\beta})$ commutes with $\varphi(x_{\alpha})$ by Item\,(b), it commutes with its inverse
$\varphi(x_{\alpha})^{-1}$. 
Hence, $[\varphi(x_{-\beta}), \varphi(x_{-(\alpha + \beta)}] = (\sigma_1 \sigma_2) (\sigma_1 \sigma_2)^{-1} = 1$.

For the third one,  we have
$[\varphi(x_{-\beta}), \varphi(x_{-(2\alpha + \beta)})] = [\sigma^{-1}_4, \sigma^{-1}_2] = 1$ by~\eqref{braid1}.

For the fourth one, conjugating by~$\Delta$, and using\,\eqref{sym2} and\,\eqref{sym4},
we reduce the desired equality 
$[\varphi(x_{-(\alpha + \beta)}), \varphi(x_{-(2\alpha + \beta)})]  = 1$
to the previous equality $[\varphi(x_{-\beta}), \varphi(x_{-(\alpha + \beta)}] = 1$.

(d) Relation~\eqref{rel-x4}: using the first equality in Lemma~\ref{lem-B6} and Equation~\eqref{rel-Delta3},
we have 
\begin{eqnarray*}
&& \hskip -36pt  [\varphi(x_{\alpha}), \varphi(x_{\beta})]^{-1} 
\varphi(x_{\alpha + \beta}) \, \varphi(x_{2\alpha + \beta}) \\
& = & (\sigma_4 \sigma_5 \sigma_4)^2 (\sigma_1\sigma_3^{-1} \sigma_5) 
(\sigma_4 \sigma_5 \sigma_4)^{-2} (\sigma_1 \sigma_3^{-1} \sigma_5) \\
& = & \Delta \, (\sigma_1 \sigma_2 \sigma_1)^2 (\sigma_1\sigma_3^{-1} \sigma_5) 
(\sigma_1 \sigma_2 \sigma_1)^{-2} (\sigma_1 \sigma_3^{-1} \sigma_5) \, \Delta^{-1}.
\end{eqnarray*}
It follows that $[\varphi(x_{\alpha}), \varphi(x_{\beta})]^{-1} 
\varphi(x_{\alpha + \beta}) \, \varphi(x_{2\alpha + \beta})$ belongs to the normal subgroup~$N$,
hence is trivial in~$B_6/N$.

(e) Relation~\eqref{rel-x5}: we use the second equality in Lemma~\ref{lem-B6} and argue as
in Item~(d). It follows that
$[\varphi(x_{\alpha}), \varphi(x_{\alpha + \beta})]^{-1} \varphi(x_{2\alpha + \beta}^2)$
is trivial in~$B_6/N$.

(f) Relation~\eqref{rel-x6a}: consider the third equality in Lemma~\ref{lem-B6}; the braid~$\gamma$
is trivial in~$B_6/N$; hence, so is
$[\varphi(x_{\alpha}), \varphi(x_{-(\alpha + \beta)})] \, \varphi(x_{-\beta}^2)$.

(g) Relation~\eqref{rel-x7}: we have
\begin{eqnarray*}
&& \hskip -28pt [\varphi(x_{\alpha}), \varphi(x_{-(2\alpha + \beta)})]^{-1} \varphi(x_{-(\alpha + \beta)})^{-1}  \varphi(x_{-\beta})\\
& = & \underbrace{\sigma_2^{-1} \sigma_ 5 \sigma_4 \sigma_ 1 \sigma_3^{-1} \sigma_5} 
\underbrace{\sigma_4^{-1} \sigma_5^{-1} 
\sigma_2 \sigma_ 5 \sigma_4} \sigma_ 1^{-1} \sigma_3 \sigma_5^{-1}   \\
&& \times \, \underbrace{\sigma_4^{-1} \sigma_5^{-1} \sigma_1 \sigma_2 \sigma_ 5 \sigma_4} 
\underbrace{\sigma_1 \sigma_3^{-1} \sigma_5} 
\underbrace{\sigma_4^{-1} \sigma_5^{-1} \sigma_2^{-1} \sigma_1^{-1} \sigma_ 4^{-1}} \\
& = & \sigma_ 5 \sigma_4  \sigma_5 \sigma_2^{-1} \sigma_ 1 \sigma_3^{-1}
\sigma_2  \sigma_ 1^{-1} \sigma_3 \underbrace{\sigma_5^{-1} \sigma_1 \sigma_2 \sigma_5}
\sigma_3^{-1} \underbrace{\sigma_1 \sigma_2^{-1} \sigma_1^{-1}} 
\underbrace{\sigma_4^{-1} \sigma_5^{-1}  \sigma_ 4^{-1}} \\
& = & \sigma_ 5 \sigma_4  \sigma_5 \sigma_2^{-1} \sigma_ 1 \sigma_3^{-1}
\sigma_2   \underbrace{\sigma_ 1^{-1} \sigma_3 \sigma_1} 
\underbrace{\sigma_2 \sigma_3^{-1} \sigma_2^{-1}} \sigma_1^{-1} \sigma_2
\sigma_5^{-1} \sigma_4^{-1}  \sigma_5^{-1} \\
& = & \sigma_ 5 \sigma_4  \sigma_5 \sigma_2^{-1} \sigma_ 1 \sigma_3^{-1}
\sigma_2   \sigma_3 
\sigma_3^{-1} \sigma_2^{-1} \sigma_3 \sigma_1^{-1} \sigma_2
\sigma_5^{-1} \sigma_4^{-1}  \sigma_5^{-1} \\
& = & (\sigma_ 5 \sigma_4  \sigma_5 \sigma_2^{-1} \sigma_ 1 \sigma_3^{-1} \sigma_2   \sigma_3)
(\sigma_ 5 \sigma_4  \sigma_5 \sigma_2^{-1} \sigma_ 1 \sigma_3^{-1} \sigma_2   \sigma_3)^{-1}  = 1.
\end{eqnarray*}

(h) Relation~\eqref{rel-x8}: we have
\begin{eqnarray*}
&& \hskip -38pt [\varphi(x_{\beta}), \varphi(x_{-(\alpha + \beta)})]^{-1} \varphi(x_{-\alpha}) \, \varphi(x_{-(2\alpha + \beta)}) \\
& = & 
\sigma_1 \sigma_2 \sigma_ 5 \sigma_4 \sigma^{-1}_1 \sigma_3 \sigma_5^{-1}
\sigma_4^{-1}\underbrace{\sigma_ 5^{-1} \sigma_2^{-1} \sigma_1^{-1} \sigma_5} 
\sigma_1 \sigma_2 \sigma_ 5 \sigma_4 \sigma_1 \sigma_3^{-1} \sigma_5 \\
&& \times \, \sigma_4^{-1}\sigma_ 5^{-1} \underbrace{\sigma_2^{-1} \sigma_1^{-1} \sigma_5^{-1}
\sigma_1 \sigma_2} \sigma_1 \sigma_3^{-1} \sigma_5\underbrace{\sigma_2^{-1} \sigma_1^{-1} \sigma_2^{-1}} \\
& = & \sigma_1 \sigma_2 \sigma_ 5 \sigma_4 \sigma^{-1}_1 \sigma_3 \sigma_5^{-1} \sigma_4^{-1}
\underbrace{\sigma_2^{-1} \sigma_1^{-1} \sigma_1 \sigma_2} \sigma_ 5 \sigma_4 \sigma_1 \sigma_3^{-1} \sigma_5 \\
&& \times \,  
\sigma_4^{-1}\sigma_ 5^{-1} \underbrace{\sigma_5^{-1} \sigma_1 \sigma_3^{-1} \sigma_5 \sigma_1^{-1}} 
\sigma_2^{-1} \sigma_1^{-1} \\
& = & \sigma_1 \sigma_2 \sigma_ 5 \sigma_4 \underbrace{\sigma^{-1}_1 \sigma_3 \sigma_5^{-1} 
\underbrace{\sigma_4^{-1}\sigma_ 5 \sigma_4} \sigma_1} \sigma_3^{-1} \sigma_5 
\sigma_4^{-1}\underbrace{\sigma_ 5^{-1} \sigma_3^{-1}}  \sigma_2^{-1} \sigma_1^{-1}\\
& = & \sigma_1 \sigma_2 \sigma_ 5 \sigma_4  \sigma_3 \underbrace{\sigma_5^{-1} \sigma_5}
\sigma_ 4 \underbrace{\sigma_5^{-1}  \sigma_3^{-1} \sigma_5}
\sigma_4^{-1}\sigma_3^{-1} \sigma_ 5^{-1} \sigma_2^{-1} \sigma_1^{-1}\\
& = & \sigma_1 \sigma_2 \sigma_ 5 \sigma_4  \sigma_3 \sigma_ 4   
\underbrace{\sigma_3^{-1} \sigma_4^{-1}\sigma_3^{-1}} \sigma_ 5^{-1} \sigma_2^{-1} \sigma_1^{-1}\\
& = & \sigma_1 \sigma_2 \sigma_ 5 \sigma_4  \sigma_3 \sigma_ 4   
\sigma_4^{-1} \sigma_3^{-1}\sigma_4^{-1} \sigma_ 5^{-1} \sigma_2^{-1} \sigma_1^{-1} \\
& = & (\sigma_1 \sigma_2 \sigma_ 5 \sigma_4  \sigma_3 \sigma_ 4)
(\sigma_1 \sigma_2 \sigma_ 5 \sigma_4  \sigma_3 \sigma_ 4)^{-1} 
= 1.
\end{eqnarray*}

(i) Relation~\eqref{rel-x9}: the fourth equality in Lemma~\ref{lem-B6} implies that the braid
$[\varphi(x_{\alpha + \beta}), \varphi(x_{-\alpha})] \, \varphi(x_{\beta}^2)$ belongs to~$N$,
hence is trivial in~$B_6/N$.

(j) Relation~\eqref{rel-x10}: we have
\begin{eqnarray*}
&& \hskip -48pt [\varphi(x_{\alpha + \beta}), \varphi(x_{-\beta})]^{-1} 
\varphi(x_{\alpha}) \, \varphi(x_{2\alpha + \beta}^{-1}) \\
& = & \sigma^{-1}_4 \underbrace{\sigma_1 \sigma_3^{-1} \sigma_5 \sigma_4 \sigma_1^{-1}} \sigma_3 
\underbrace{\sigma_5^{-1} \sigma_5 }
\sigma_4 \underbrace{\sigma_1 \sigma^{-1}_3 \sigma_5 \sigma_4^{-1} \sigma_5^{-1} \sigma_1^{-1}} \\
& = & \sigma^{-1}_4 \sigma_3^{-1} \sigma_5 \underbrace{\sigma_4  \sigma_3 \sigma_4 }
 \sigma^{-1}_3 \underbrace{\sigma_5 \sigma_4^{-1} \sigma_5^{-1}} \\
& = & \sigma^{-1}_4 \underbrace{\sigma_3^{-1} \sigma_5 \sigma_3} 
 \underbrace{\sigma_4 \sigma_3  \sigma^{-1}_3 \sigma_4^{-1}} \sigma_5^{-1} \sigma_4 \\
& = & \sigma^{-1}_4 \sigma_5 \sigma_5^{-1} \sigma_4 = 1.
\end{eqnarray*}

(k) Relation~\eqref{rel-x11}: using the computation in Item~(j) and conjugating by~$\Delta$,
we obtain
\begin{eqnarray*}
&& \hskip -48pt 
[\varphi(x_{\alpha + \beta}), \varphi(x_{-(2\alpha + \beta)})]^{-1}  \varphi(x_{-\alpha}) \, \varphi(x_{\beta}^{-1}) \\
& = & \Delta\, \, [\varphi(x_{\alpha + \beta}), \varphi(x_{-\beta})]^{-1}  \varphi(x_{\alpha}) \, \varphi(x_{2\alpha + \beta}^{-1}) 
\, \Delta^{-1}
= 1.
\end{eqnarray*}

(l) Relation~\eqref{rel-x12}: by \eqref{sym1}, \eqref{sym3} and \eqref{sym4} we have
\begin{multline*}
[\varphi(x_{2\alpha + \beta}), \varphi(x_{-\alpha})] \, \varphi(x_{\alpha + \beta}) \, \varphi(x_{ \beta})\\
= \Delta\, [\varphi(x_{\beta}), \varphi(x_{\alpha})] \, \varphi(x_{\alpha + \beta}) \, \varphi(x_{2\alpha + \beta})\, \Delta^{-1} ,
\end{multline*}
which in view of Item~(d) above implies that
$[\varphi(x_{2\alpha + \beta}), \varphi(x_{-\alpha})] \, \varphi(x_{\alpha + \beta}) \, \varphi(x_{ \beta})$
is trivial in~$B_6/N$.

(m) Relation~\eqref{rel-x13} reduces to Relation~\eqref{rel-x8} after conjugating by~$\Delta$.
In the same way Relations~\eqref{rel-x14} and~\eqref{rel-x15} reduce to 
Relations~\eqref{rel-x7} and~\eqref{rel-x6a} respectively. 

To complete the proof it remains to check that $\varphi \circ f  = \id$.
This is a consequence of the equalities 
$(\varphi \circ f )(\sigma_i) = \sigma_i$ ($1\leq i \leq 5$),
which follow immediately from the definitions of~$f$ and~$\varphi$.
\end{proof}

As can be seen in the previous proof, 18~out of the 24~images under $\varphi$ of the relations 
defining~$\St(C_2,\ZZ)$ hold in the braid group~$B_6$ itself. 
Only six of them hold modulo the normal subgroup~$N$.

\begin{rem}\label{rem-Sp}
In \cite[Cor.~7]{Ka} we gave a braid-like presentation of the symplectic modular group~$\Sp_4(\ZZ)$
with four additional relations rather than the two additional ones in Corollary~\ref{coro-Sp} above.
In~\cite{Ka} the first relation in Equation~(27) is equivalent (modulo the braid relations)
to the relation $(\sigma_1 \sigma_2 \sigma_ 1)^4 = 1$ above.
Relation~(28) in \emph{loc. cit.} is equivalent to the relation
$(\sigma_1 \sigma_2 \sigma_1)^2 (\sigma_1 \sigma_3^{-1} \sigma_5) (\sigma_1 \sigma_2 \sigma_1)^{-2} 
= (\sigma_1 \sigma_3^{-1} \sigma_5)^{-1}$. 
Our corollary~\ref{coro-Sp} shows that the relation $\Delta^2 = 1$ of~\cite[Eq.~(26)]{Ka}
is not needed. Actually, one can prove
\begin{equation*}
f(\Delta^2) = w_{\beta}^{12}   = w_{2\alpha+ \beta}^{12} 
= f\left((\sigma_4 \sigma_5 \sigma_4)^{12} \right)
= f\left((\sigma_1 \sigma_2 \sigma_1)^{12} \right).
\end{equation*}
\end{rem}

%%%%%% Appendix

\appendix

%%%
\section{The Steinberg group of type~$C_2$ over a commutative ring}\label{app-St}

For completeness we give a presentation of the Steinberg group $\St(C_2,R)$ 
over an arbitrary commutative ring~$R$.

By~\cite[Sect.\,3]{StM} (see also~\cite{St0} or~\cite[Chap.\,6]{St}) the group~$\St(C_2,R)$ has a presentation 
with the set of generators $\{ x_{\gamma}(u) : \gamma \in \Phi, \, u\in R\}$
subject to the following relations:
\begin{itemize}
\item
$x_{\gamma}(u+v) = x_{\gamma}(u) x_{\gamma}(v)$ for all $u,v \in R$;

\item
if $\gamma, \delta \in \Phi$ such that $\gamma + \delta \neq 0$, then
\begin{equation*}
\left[ x_{\gamma}(u), x_{\delta}(v) \right] = \prod\, x_{i\gamma+j\delta}(c_{i,j}^{\gamma, \delta} u^i v^j),
\end{equation*}
where $i$ and $j$ are positive integers such that $i\gamma+j\delta$ belongs to~$\Phi$. 
\end{itemize}
The coefficients $c_{i,j}^{\gamma, \delta}$ are the integers which have already come up in
Relations~\eqref{rel-x} in the proof of Proposition~\ref{prop-St}.

In this way we obtain the following 24 defining relations for~$\St(C_2,R)$ 
(where $u,v \in R$):
\begin{multline}\label{rel-B1}
[x_{\alpha}(u), x_{2\alpha + \beta}(v)]  = [x_{\beta}(u), x_{\alpha + \beta}(v)] = \\
= [x_{\beta}(u), x_{2\alpha + \beta}(v)] = [x_{\alpha + \beta}(u), x_{2\alpha + \beta}(v)]  = 1, 
\end{multline}
\begin{multline}\label{rel-B2}
[x_{\alpha}(u), x_{- \beta}(v)] = [x_{\beta}(u), x_{-\alpha}(v)] = \\
= [x_{\beta}(u), x_{-(2\alpha + \beta)}(v)] = [x_{-\beta}(u), x_{2\alpha + \beta}(v)] = 1 , 
\end{multline}
\begin{multline}\label{rel-B3}
[x_{-\alpha}(u), x_{-(2\alpha + \beta)}(v)] = [x_{-\beta}(u), x_{-(\alpha + \beta)}(v)] = \\
= [x_{-\beta}(u), x_{-(2\alpha + \beta)}(v)] 
=  [x_{-(\alpha + \beta)}(u), x_{-(2\alpha + \beta)}(v)] = 1, 
\end{multline}
\begin{equation}\label{rel-B4}
[x_{\alpha}(u), x_{\beta}(v)] = x_{\alpha + \beta}(uv) \,  x_{2\alpha + \beta}(u^2v)  , 
\end{equation}
\begin{equation}\label{rel-B5}
[x_{\alpha}(u), x_{\alpha + \beta}(v)] = x_{2\alpha + \beta}(2uv) , 
\end{equation}
\begin{equation}\label{rel-B6}
[x_{\alpha}(u), x_{-(\alpha + \beta)}(v)] = x_{-\beta}(-2uv)  , 
\end{equation}
\begin{equation}\label{rel-B7}
[x_{\alpha}(u), x_{-(2\alpha + \beta)}(v)] = x_{- \beta}(u^2v) \, x_{-(\alpha + \beta)}(-uv) ,
\end{equation}
\begin{equation}\label{rel-B8}
[x_{\beta}(u), x_{-(\alpha + \beta)}(v)] = x_{-\alpha}(uv) \, x_{-(2\alpha + \beta)}(uv^2)  ,
\end{equation}
\begin{equation}\label{rel-B9}
[x_{\alpha + \beta}(u), x_{-\alpha}(v)] = x_{\beta}(-2uv) , 
\end{equation}
\begin{equation}\label{rel-B10}
[x_{\alpha + \beta}(u), x_{-\beta}(v)] = x_{\alpha}(uv) \, x_{2\alpha + \beta}(-u^2v)   ,
\end{equation}
\begin{equation}\label{rel-B11}
[x_{\alpha + \beta}(u), x_{-(2\alpha + \beta)}(v)] = x_{-\alpha}(uv) x_{\beta}(-u^2v)  ,
\end{equation}
\begin{equation}\label{rel-B12}
[x_{2\alpha + \beta}(u), x_{-\alpha}(v)] = x_{\alpha + \beta}(-uv) \, x_{ \beta}(-uv^2)   ,
\end{equation}
\begin{equation}\label{rel-B13}
[x_{2\alpha + \beta}(u), x_{-(\alpha + \beta)}(v)] = x_{\alpha}(uv) x_{-\beta}(uv^2)   ,
\end{equation}
\begin{equation}\label{rel-B14}
[x_{-\alpha}(u), x_{-\beta}(v)] = x_{-(\alpha + \beta)}(-uv) \,  x_{-(2\alpha + \beta)}(u^2v)    ,
\end{equation}
\begin{equation}\label{rel-B15}
[x_{-\alpha}(u), x_{-(\alpha + \beta)}(v)] = x_{-(2\alpha + \beta)}(-2uv)  . 
\end{equation}

%%%
\section*{Acknowledgement}
I am indebted to Fran\c cois Digne for Remark~\ref{rem-Delta}.

%%%


\begin{thebibliography}{99}

\bibitem{Be1}
H.~Behr,
\emph{Eine endliche Pr\"asentation der symplektischen Gruppe~$\Sp_4(\ZZ)$},
{Math.~Z.}~{141} (1975), 47--56.

\bibitem{Be2}
H.~Behr,
\emph{Explizite Pr\"asentation von Chevalleygruppen \"uber~$\ZZ$},
{Math.~Z.}~{141} (1975), 235--241.

\bibitem{Bn}
P.~Bender, \emph{Eine Pr\"asentation der symplektischen Gruppe~$\Sp_4(\ZZ)$
mit 2~Erzeugenden und $8$~definierenden Relationen}, {J.~Algebra}~{65} (1980), 328--331.

\bibitem{BDM}
D.~Bessis, F.~Digne, J.~Michel, 
\emph{Springer theory in braid groups and the Birman-Ko-Lee monoid}, Pacific J. Math. 205 (2002), 287--309;
DOI: 10.2140/pjm.2002.205.287. 

\bibitem{Bi} 
J. S. Birman, \emph{Braids, links and mapping class groups},
Annals of Math.\ Studies, No.~82 (Princeton University Press, Princeton,~1975).

\bibitem{Bo}
N.~Bourbaki, 
\emph{\'El\'ements de math\'ematique. Groupes et alg\`ebres de Lie. Chapitres~IV, V, VI},
Actualit\'es Scientifiques et Industrielles, Hermann, Paris,~1968.
English translation: \emph{Lie groups and Lie algebras}, Springer-Verlag, Berlin,~2002.

\bibitem{Ca}
R.~Carter, 
\emph{Simple groups of Lie type}, John Wiley \& Sons, London-New York-Sydney,~1972.

\bibitem{De}
P.~Dehornoy, 
\emph{Le calcul des tresses}, Calvage \& Mounet, Paris,~2019.

\bibitem{DDRW}
P.~Dehornoy, I.~Dynnikov, D.~Rolfsen, B.~Wiest, \emph{Ordering braids}, 
Math. Surveys Monogr.,~148, American Mathematical Society, Providence, RI,~2008;
DOI: dx.doi.org/10.1090/surv/148. 

\bibitem{E++}
D. B. A. Epstein, J. W. Cannon, D. F. Holt, S. V. F. Levy, M. S. Paterson, W. P. Thurston, 
\emph{Word processing in groups}, Jones and Bartlett Publishers, Boston, MA,~1992.

\bibitem{HR}
L. K. Hua, I. Reiner,
\emph{On the generators of the symplectic modular group},
Trans. Amer. Math. Soc. 65 (1949), 415--426. 

\bibitem{Ka}
C.~Kassel,
\emph{On an action of the braid group $B_{2g+2}$ on the free group~$F_{2g}$}, 
Internat.\ J.~Algebra Comput.\ 23, No.~4 (2013), 819--831;
DOI: 10.1142/S0218196713400110.

\bibitem{KT}
C.~Kassel, V.~Turaev, \emph{Braid groups},
Grad.\ Texts in Math., Vol.~247, Springer, New York,~2008;
DOI: 10.1007/978-0-387-68548-9.

\bibitem{Ma}
H.~Matsumoto, 
\emph{Sur les sous-groupes arithm\'etiques des groupes semi-simples d\'eploy\'es},
Ann. Scient. \'Ec. Norm. Sup. 4e~s\'erie,~2 (1969), 1--62.

\bibitem{StM}
M. R. Stein,
\emph{Generators, relations and coverings of Chevalley groups over commutative rings},
Amer. J. Math. 93 (1971), 965--1004. 

\bibitem{St0}
R.~Steinberg, 
\emph{G\'en\'erateurs, relations et rev\^etements de groupes alg\'ebriques}. 
1962 Colloq. Th\'eorie des Groupes Alg\'ebriques (Bruxelles, 1962) pp. 113--127. 
Librairie Universitaire, Louvain; Gauthier-Villars, Paris.

\bibitem{St}
R.~Steinberg, 
\emph{Lectures on Chevalley groups}. Notes prepared by John Faulkner and Robert Wilson. 
Revised and corrected edition of the 1968 original.
University Lecture Series, 66. American Mathematical Society, Providence, RI, 2016.


\end{thebibliography}
\end{document}